\documentclass{amsart}
\usepackage{amssymb,amsfonts,latexsym}

\newtheorem{theorem}{Theorem}[section]
\newtheorem{lemma}[theorem]{Lemma}
\newtheorem{proposition}[theorem]{Proposition}

\theoremstyle{definition}
\newtheorem{definition}[theorem]{Definition}

\newtheorem{remark}[theorem]{Remark}
\newtheorem{general remarks}[theorem]{General remarks}

\newcommand{\Id}{\operatorname{Id}}

\newcommand{\ben}{\begin{enumerate}}
\newcommand{\een}{\end{enumerate}}

\begin{document}

\title[Simple $G$- graded  algebras and their polynomial identities]{Simple $G$-graded algebras and their polynomial identities}

\author{Eli Aljadeff}

\author{Darrell Haile}


\maketitle

\centerline{\today}

\begin{abstract}

Let $G$ be any group and $F$ an algebraically closed field of
characteristic zero. We show that any $G$-graded finite
dimensional associative $G$-simple algebra over $F$ is determined
up to a $G$-graded isomorphism by its $G$-graded polynomial
identities. This result was proved by Koshlukov and Zaicev in case
$G$ is abelian.

\end{abstract}

\bigskip
\section*{Introduction}

The purpose of this article is to prove that finite dimensional (associative) simple $G$-graded algebras over an algebraically
closed field $F$ of characteristic zero are determined up to
$G$-graded isomorphism by their $G$-graded identities. Here $G$ is
any group. In case $G$ is abelian, the result was established by Koshlukov and Zaicev
\cite{KZ}. Analogous results were obtained for Lie algebras by Kushkulei and Razmyslov \cite{KR}
and for Jordan algebras by Drensky and Racine \cite{DR}   and recently for nonassociative algebras by Shestakov and Zaicev \cite{SZ}.

The structure theory of finite dimensional $G$-graded algebras and
in particular of simple $G$-graded  algebras plays a crucial role
in the proof of the representability theorem for $G$-graded PI
algebras and in the solution of the Specht problem (that is, that  the
$T$-ideal of $G$-graded identities is finitely based) for such
algebras (see Aljadeff and Kanel-Belov \cite{AB}).

Recall that the representability theorem for $G$-graded algebras
says in particular, that if $W$ is an affine $G$-graded algebra
which is PI as an ordinary algebra, then there exists a finite
dimensional algebra $A$ which satisfies precisely the same
$G$-graded identities as $W$.

A fundamental part of the proof of the representability theorem is
the construction of special  finite dimensional $G$-graded
algebras which are called basic. It turns out that if $B$ is a
basic algebra, then $B$ admits $G$-graded polynomials which are
called Kemer. These are multilinear polynomials, nonidentities,
which admit alternating sets of homogeneous elements of degree $g
\in G$ whose cardinalities are maximal possible \cite{AB}. In the
special case where the basic algebra has no radical, it is in fact
$G$-simple, and in that case, the cardinalities of the alternating
sets in a Kemer polynomial coincide with the dimensions of the
homogeneous components. Clearly, no nonidentity polynomial of a
$G$-simple algebra (or in fact of any $G$-graded algebra) can have larger alternating sets.

 The key point in the proof of representability is that the finite
dimensional algebra $A$ which satisfies the same $G$-graded
identities as $W$ can be expressed as the direct sum of basic
algebras and hence the $T$-ideal of $G$-graded identities of $A$
is the intersection of the corresponding ideals of identities of
the basic algebras which appear in the decomposition. However, the
basic algebras that appear in the decomposition of $A$ are not
known to be unique. In fact, even the basic algebras themselves
are not determined in general by their identities (as a result of
the interaction of the simple components via the radical). But if
the basic algebra is $G$-semisimple (and hence $G$-simple), the
main result of the paper says that in that case the answer is
positive.

To state the result precisely we recall some basic definitions.
Let $k$ be an arbitrary field and let $G$ be a group.   A
$k$-algebra $A$ is said to be $G$-graded if for each $g\in G$
there is a  $k$-subspace $A_g$ of $A$ (possibly zero) such that
for all $g,h\in G$, we have $A_gA_h\subseteq A_{gh}$.   Such a
$G$-graded algebra is said to be a simple $G$-graded algebra (or a
$G$-simple algebra) if there are no nontrivial homogeneous ideas,
or equivalently if the ideal generated by each nonzero homogeneous
element is the whole algebra.

A $G$-graded polynomial is a polynomial in the free
algebra $k \langle X_G \rangle$ where $X_G$ is the union of sets
$X_g$, $g \in G$ and $X_g = \{x_{1,g}, x_{2,g},...\}$. In other
words, $X_G$ consists of countably many variables of degree $g$
for every $g \in G$. We say that a polynomial
$p(x_{1,g_{i_1}},...,x_{n,g_{i_n}})$ in $k \langle X_G \rangle$ is
a $G$-graded identity of a $G$-graded algebra $A$ if $p$ vanishes
for every graded evaluation on $A$. The set of $G$-graded
identities of $A$ is an ideal of $k \langle X_{G} \rangle$ which
we denote by $Id_{G}(A)$. Moreover it is a $T$-ideal, that is,   it is
closed under $G$-graded endomorphisms of $k \langle X_G \rangle$.

It is known that if $k$ has characteristic zero, the $T$-ideal of
identities is generated as a $T$-ideal by multilinear polynomials,
that is  graded polynomials whose monomials are permutation of
each other (up to a scalar from the field).  Moreover we may
assume in addition that all of the monomials have the same
homogeneous degree. We can now state the main result of the paper.

\medskip

\noindent THEOREM:
Let $A$ and $B$ be two finite dimensional simple $G$-graded algebras over $F$ where $F$ is an
algebraically closed field of characteristic zero.  Then $A$ and $B$ are $G$-graded isomorphic if and
only if $Id_{G}(A)=Id_{G}(B)$.

\medskip

A key ingredient in the proof is the result of  Bahturin, Sehgal
and Zaicev (\cite{BSZ}, Theorem \ref{Bahturin-Sehgal-Zaicev}) that
determines the structure of a simple $G$-graded algebra as a
combination of a fine graded algebra and an elementary graded
algebra. In section $1$ we state this result and use it to define
the notion of a \underbar{presentation} of the given $G$-simple
algebra.  It is a consequence of our main theorem that any two
graded isomorphic $G$-simple algebras have equivalent
presentations.   However one can prove this uniqueness result
directly, without the use of identities and we present such a
proof in the last section of this paper.

Another  motivation for studying $G$-graded polynomial identities
of finite dimensional $G$-simple algebras is the possible existence of a ``versal"
object. It is well known that if $A$ is the algebra of $n \times
n$-matrices, the corresponding algebra of generic elements has a
 central localization which is an Azumaya algebra and is
versal with respect to all $k$-forms (in the sense of Galois
descent) of $A$ where $k$-is any field of zero characteristic.
Furthermore, extending the center to the field of fractions, one
obtains a division algebra, the so called generic division
algebra, which is a form of $A$. The algebra of generic elements
can be constructed in a different way. It is well known that it is
isomorphic to the the relatively free algebra of $A$, namely, the
free algebra on a countable set of variables modulo the $T$-ideal
of identities.

 Given a
$G$-graded finite dimensional algebra one can construct the
corresponding $G$-graded relatively free algebra, and it is of
interest to know whether there exists a versal object in this case
as well. It turns out that this is so for some specific cases as
in \cite{AHN} and \cite{AKar} and \cite{AKassel}.

Clearly, if two nonisomorphic finite dimensional $G$-simple algebras $A$ and $B$
had the same $T$-ideal of identities, there could not exist a versal object for $A$  (or $B$). So in
view of  our main theorem,  it is natural to ask whether  for an
arbitrary finite dimensional $G$-simple algebra there exists a corresponding versal
object.

\section{Preliminaries}

We start by recalling some terminology. Let $G$ be any group and
$A$ a finite dimensional simple $G$-graded algebra. As mentioned
in the introduction, our proof is based on a result of Bahturin,
Sehgal and Zaicev\cite{BSZ} in which they present any finite dimensional $G$-graded
simple algebra by means of two type of $G$-gradings,  fine
and elementary. Before stating their theorem let us give two
examples, one of each kind.

Given a finite subgroup $H$ of $G$ we can consider the group
algebra $FH$ with the natural $H$-grading. This algebra is
$H$-simple, in fact an $H$-division algebra in the sense that every nonzero homogeneous element is invertible.
 Moreover we can view the algebra $FH$ as a $G$-graded
algebra where the $g$-homogeneous component is set to be $0$ if
$g$ is not in $H$. More generally we may consider any twisted
group algebra $F^{\alpha}H$, where $\alpha$ is a $2$-cocycle of
$H$ with invertible values in $F$, again as a $G$-graded algebra.
As in the case where the cocycle is trivial, the algebra
$F^{\alpha}H$ is a finite dimensional $G$-division algebra.  We refer to such a grading as a fine
grading. The second type of grading is called elementary. Let
$M_{r}(F)$ be the algebra of $r \times r$ matrices over the field
$F$. Fix an $r$-tuple $(p_{1},...,p_{r}) \in G^{(r)}$, and assign the
elementary matrix $e_{i,j}$, $1 \leq i,j \leq n$ the homogeneous
degree $p_{i}^{-1}p_{j}$. Note that the product of the elementary
matrices is compatible with their homogeneous degrees and so we
obtain a $G$-grading on $M_{r}(F)$. Furthermore, since $M_{r}(F)$
is a simple algebra it is also $G$-simple.

The result of Bahturin, Sehgal and Zaicev\cite{BSZ} is the graded
version of Wedderburn's structure theorem for finite dimensional
simple algebras.  Their result says that every finite dimensional
$G$-simple algebra is isomorphic to a $G$-graded algebra which is
the tensor product of two $G$-simple algebras, one with fine
grading (hence a graded division algebra)  and the other a full
matrix algebra with an elementary grading. Here is the precise
statement.

\begin{theorem} \label{Bahturin-Sehgal-Zaicev}
\cite{BSZ}Let $A$ be a finite dimensional $G$-simple algebra over
an algebraically closed field $F$ of characteristic zero. Then
there exists a finite subgroup $H$ of $G$, a $2$-cocycle
$\alpha:H\times H\rightarrow F^{*}$ where the action of $H$ on $F$
is trivial, an integer $r$ and a $r$-tuple
$(p_{1},p_{2},\ldots,p_{r})\in G^{(r)}$ such that $A$ is
$G$-graded isomorphic to $C=F^{\alpha}H\otimes M_{r}(F)$ where
$C_{g}=span_{F}\{u_{h}\otimes e_{i,j}:g=p_{i}^{-1}hp_{j}\}$. Here
$u_{h}\in F^{\alpha}H$ is a representative of $h\in H$ and
$e_{i,j}\in M_{r}(F)$ is the $(i,j)$ elementary matrix.

In particular the idempotents $1\otimes e_{i,i}$ as well as the
identity element of $A$ are homogeneous of degree $e\in G$.
\end{theorem}

\begin{definition} Given a finite dimensional $G$-simple algebra $A$, let $H$, $\alpha \in Z^{2}(H,F^{*})$ and $(p_1,\ldots,p_r) \in G^{(r)}$ be as in the theorem above.
We denote the triple $(H,
\alpha, (p_1,\ldots,p_r))$ by $P_{A}$ and refer to it as a
\underbar{presentation} of the $G$-graded algebra $A$.  We will refer to $r$ as the matrix size of $P_A$.

\end{definition}

Clearly, a presentation determines the $G$-graded structure of $A$
up to a $G$-graded isomorphism. On the other hand, a $G$-graded
algebra may admit more than one presentation and so we need to
introduce a suitable equivalence relation on presentations.

We start by establishing some conditions on presentations which
yield $G$-graded isomorphic algebras.

\begin{lemma}
Let $A$  be a finite dimensional $G$-simple algebra with presentation
$P_{A}=(H, \alpha, (p_1,\ldots,p_r)).$  The following ``moves'' (and their
composites) on the presentation determine $G$-graded algebras
$G$-graded isomorphic to $A$.

\begin{enumerate}

\item

Permuting the $r$-th tuple, that is $A^{'} \cong
F^{\alpha_{A}}H_{A}\otimes_{F}M_{r}(F)$ and the elementary grading is
given by $(p_{\pi(1)},...,p_{\pi(r)})$ where $\pi \in Sym(r)$.

\item

Replacing any entry $p_i$ of $(p_1,...,p_r)$ by any element
$h_{0}p_{i} \in Hp_{i}$ (changing right $H$-coset
representatives).

\item

For an arbitrary $g\in G$,

\begin{enumerate}

\item  replacing $H$ with the conjugate $H^{g}=gHg^{-1}$,

\item replacing
the cocycle $\alpha$ by $\alpha^g$ where
$$ \alpha^g(gh_{1}g^{-1}, gh_{1}g^{-1})=\alpha(h_1,h_2)$$ and

 \item  shifting the tuple $(p_1,...,p_r)$ by $g$, that is,  replacing the tuple $(p_1,...,p_r)$ by $(gp_1,...,gp_r)$.

\end{enumerate}

\end{enumerate}
\end{lemma}

\begin{proof}

We describe the isomorphism maps.

(1)
$$ u_{h}\otimes e_{k,l} \longmapsto u_{h}\otimes e_{\pi(k),\pi(l)} $$

(2)
$$u_{h}\otimes e_{k,l} \longmapsto u_{h}\otimes e_{k,l}$$ if $k \neq i$ and $l\neq i$.
$$u_{h}\otimes e_{i,l} \longmapsto u_{h_{0}}u_{h}\otimes e_{i,l}$$ if $l\neq i$.
$$u_{h}\otimes e_{k,i} \longmapsto u_{h}u^{-1}_{h_{0}}\otimes e_{k,i}$$ if $k \neq i$.
$$u_{h}\otimes e_{i,i} \longmapsto u_{h_{0}}u_{h}u^{-1}_{h_{0}}\otimes e_{i,i}$$

(3)
$$ u_{h}\otimes e_{k,l} \longmapsto u_{ghg^{-1}}\otimes e_{k,l} $$

We leave the reader the task of showing that  these maps are indeed
isomorphisms.
\end{proof}

We will call these isomorphisms \underbar{basic} \underbar{moves} of type (1), (2), or (3).
We will call presentations $P_{A}$ of the $G$-simple algebra $A$ and
$P_{B}$ of the $G$-simple algebra $B$ equivalent if one
is obtained from the other by a (finite) sequence of basic moves.
This is clearly an equivalence relation on presentations.
It follows from the lemma that algebras with equivalent
presentations are $G$-graded isomorphic.

Let $A$ be $G$-simple with presentation $P_A$.  Our proof
requires, in terms of the given presentation $P_A$,  a rather
precise understanding of the structure of the subalgebra
$A_{N}=\sum_{g \in N} A_{g}$ (of $A$) where $N$ is an arbitrary
subgroup of $G$.  To this end we introduce an equivalence relation
on the elements of the $r$-tuple $(p_1,...,p_r)$:  We will say
 $i,j \in \{1,...,r\}$  are
$N$-\textit{related} in $P_{A}$  if there exists $h\in
H_A$ such that $p^{-1}_ihp_{j} \in N$. It is easy to see that this
is indeed an equivalence relation. We may assume (after permuting
the elements of the tuple $(p_1,...,p_r)$ if needed) that the
tuple is decomposed into subtuples whose elements are the
corresponding equivalence classes. We denote the classes by
$(p_{i_1},p_{i_1+1},...,p_{i_1+k_{1}-1})$,
$(p_{i_2},p_{i_2+1},...,p_{i_2+k_{2}-1})$,...,
$(p_{i_d},p_{i_d+1},...,p_{i_d+k_{d}-1})$.

In order to get a better understanding of the $N$-elements in the
presentation $P_{A}$, we focus our attention on one equivalence
class, say $(p_{i_1},p_{i_1+1},...,p_{i_1+k_{1}-1})$, and so, for
convenience we change the notation by letting $k=k_1$  and setting
$(g_1,...g_k) = (p_{i_1},p_{i_1+1},...,p_{i_1+k_{1}-1})$.  We let
$A_{N,1}$ denote the $F$-space spanned by the elements $u_h\otimes
e_{i,j}$ where $i,j \in \{1,...,k\}$ and $g_i^{-1}hg_j$ is in $N$.

For $i=1,...,k$ we consider the following subgroup of $N$,

$$ \Omega_{g_{i}}=g_{i}^{-1}Hg_{i} \cap N$$ and let $d_i$ be its order.

\begin{proposition} \label{structure of N-graded subalgebra}

With the notation as above, the following hold.

\begin{enumerate}

\item

For $1\leq i,j \leq k$ the subgroups $\Omega_{g_{i}}$ and
$\Omega_{g_{j}}$ are conjugate to each other by an element of $N$.
In particular $d_i=d_j $.

\item

For $i,j \in \{1,...,k\}$ the set

$$g_{i}^{-1}Hg_{j} \cap N$$

is a left $\Omega_{g_{i}}$-coset and a right
$\Omega_{g_{j}}$-coset . In particular the order of
$$g_{i}^{-1}Hg_{j} \cap N$$ is $d_i(=d_j)$.

\item

The subalgebra $A_{N,1}$ is $G$-simple with presentation   $$P_{A_{N,1}}=( N \cap
g_{1}^{-1}Hg_{1}, g_1(\alpha), (n_1,...,n_k))$$
for some
elements $n_1,...,n_k$, where $n_j \in N \cap
g_{1}^{-1}Hg_{j}$.

\end{enumerate}

\begin{proof}

This is straightforward.  We will prove only
the first statement. By the equivalence condition, there are
elements $h \in H$ and $n \in N$ such that $g_{i}^{-1}hg_{j}=n$.
Hence
$$ \Omega_{g_{i}}=g_{i}^{-1}Hg_{i} \cap N=$$
$$ng_{j}^{-1}h^{-1}Hhg_{j}n^{-1} \cap N=$$
$$n(g_{j}^{-1}Hg_{j}\cap N)n^{-1}=$$
$$ n(\Omega_{g_{i}})n^{-1}$$
as desired.

\end{proof}

\end{proposition}

\begin{remark} \label{blocks and pages}

Based on the presentation of the $N$-simple algebra above, we see
that the appearance of an $N$-simple component constitutes of a
diagonal block of the $r \times r$-matrix algebra.  We will refer to the number $d_i$ as the number of
\textit{pages} in that component. So each $N$-simple component sits on the diagonal with a
certain matrix size and a certain number of pages.

\end{remark}

\section{Proofs}

Our aim is to show that algebras $A$ and $B$ (finite dimensional
and $G$-simple) with nonequivalent presentations $P_{A}$ and
$P_{B}$ have different $T$-ideals of $G$-graded identities and
hence are $G$-graded nonisomorphic. This will imply

\begin{enumerate}

\item

$G$-graded (finite dimensional) $G$-simple algebras $A$ and $B$ are $G$-graded
isomorphic if and only if any two presentations $P_{A}$ and
$P_{B}$ are equivalent.

\item

$G$-graded (finite dimensional) $G$-simple algebras are characterized (up to $G$-graded
isomorphism) by their T-ideal of $G$-graded identities.

\end{enumerate}

\begin{remark} In section 3 we will give a proof of statement (1) that does not depend on identities. \end{remark}

Generally speaking, we proceed step by step where in each step we
show that if $A$ and $B$ satisfy the same $G$-graded identities,  then the presentations $P_{A}$ and $P_{B}$ must coincide on
certain ``invariants/parameters'' up to applications of basic
moves.

Let us start by exhibiting a list of such invariants of a
presentation

$$P_{A}=(H_{A}, \alpha, (p_1,\ldots,p_r))$$
of an algebra $A$.

\begin{enumerate}

\item

The dimensions of the homogeneous components $A_g$, for all $g\in G$ (and so, in particular, the dimension of $A$).

\item

The multiplicities of right $H$-coset representatives in the $r$-tuple $(p_1,...,p_r)$.

\item

The order of $H$.

\item

The group $H$ up to conjugation.

\item

The group $H.$

\bigskip

Based on (4), for the rest of the invariants we will assume the subgroup $H$ is determined.
The next sequence of invariants is determined by the $r$-tuple
$T=(p_1,...,p_r)$. We decompose $T$ into
subtuples where each subtuple
consists of all elements in $(p_1,...,p_r)$ lying in the same right coset $N(H)g$ of the normalizer of $H$ in $G$.

Let us denote the full tuple by $T$ and the subtuples by
$$T_{1}e, T_{2}g_{2},...,T_{k}g_{k}.$$

Each $T_{i}$ consists of  representatives $\sigma_{i,j}$ of $H$ in $N(H)$ with multiplicity $d_{i,j}$.

\item

The vector of multiplicities of representatives in each $T_{i}$.

\item

The coset representatives $\{g_1=e,g_2,\dots, g_k\}$ of $N(H)$ in $G$ that appear in the tuples, with multiplicities.

\item

The  elements of $T$ up to left multiplication by an element of
$N(H)$. Note that by the basic moves this determines the
presentation up to the 2-cocycle on $H$.

\bigskip

For the rest of the invariants we will assume the subgroup $H$ and the tuple $T$ are determined.

\bigskip

\item

The $2$-cocycle on $H$ up to conjugation by an
element of $N(H)$.

\bigskip

For each element $t_{i,j} \in T_{i}$ we consider the cocycle on
$H$ obtained by conjugation of $\alpha$ by $t_{i,j}^{-1}$ (note that
conjugating with $t_{i,j}g_{i}$ gives a cocycle on $H^{g_{i}^{-1}}$).
Then each $T_{i}$ determines a set of cocycles (on $H$).

\item

The set of cocycles (with multiplicities!) as determined by the
elements of $T_{i}$.

\bigskip

\noindent And then finally
\item

The presentation $P_{A}$ of $A$.

\end{enumerate}

\bigskip

We will refer to this list of steps as the outline of the proof.

Let $A$ and $B$ be $G$-graded algebras, finite dimensional $G$-simple with
presentations

$$P_{A}=(H_{A}, \alpha, (p_1,\ldots,p_r))$$ and
$$P_{B}=(H_{B}, \beta, (q_1,\ldots,q_s)).$$

Suppose $A$ and $B$ satisfy the same $G$-graded identities. Our
task will be to add (in each step) an invariant from the
list above on which the presentations $P_{A}$ and $P_{B}$ must coincide
(up to basic moves). The basic strategy  is to establish a suitable
connection between the invariants described above and the
structure of some extremal $G$-graded nonidentities of $A$.
But more than that. The polynomials we construct will establish a
strong connection between the  invariants and the
structure of \textit{any} nonzero evaluation of them (with a suitable basis).

\begin{remark}
Given a presentation $P_{A}$ of an algebra $A$, it is well known
that in order to test whether a $G$-graded multilinear polynomial
is an identity of $A$ it is sufficient to consider evaluations on
any $G$-graded basis of $A$ and so, from now on,  {\it we always
choose the basis consisting of the elements}  $u_{h}\otimes
e_{i,j}$, for all $h\in H$ and all $i,j$.   We will refer to this
basis as the standard basis for $A$ (Of course it really depends
on the presentation of $A$).   This will play a key role in the
proof since the \textit{connection} we make via nonzero
evaluations between the structure of $A$ and structure of the
polynomials will be based precisely on that particular $G$-graded
basis of $A$. In particular, all subspaces we consider will be
linear spans of subsets of that basis.
\end{remark}

We want to be more precise about what we mean by polynomials that
establish a strong connection between their nonzero evaluations
and the $G$-graded structure of $A$.   Let $V = \oplus_{g}V_{g}
\subseteq \oplus_{g}A_{g}$ be a $G$-graded subspace of $A$. Let
$d_g= \dim_{F}(A_g)$ and $\delta_{g}=\dim_{F}(V_g)$, $g \in G$. We
say that a multilinear $G$-graded polynomial $p$
\textit{allocates} the $G$-graded subspace $V$ of $A$ if the
following hold:

\begin{enumerate}

\item

$p=p(Z_{G})$ is obtained from a single multilinear monomial
$Z_{G}$ by homogeneous multialternation. This means that we choose
disjoint sets of homogeneous variables in $Z_{G}$(each set
constitute of elements of the \textit{same} homogeneous degree in
$G$) and we alternate the elements of each set successively.

\item

For every $g \in G$ with $V_{g}\neq 0$, we have a subset $T_{g}$
of $g$-variables in $Z_{G}$ of cardinality $d_{g}$, and a subset
$S_{g}$ of $T_{g}$ of cardinality $\delta_{g}$.

\item

The set
$T_{g}$ is alternating on $p(Z_{G})$, for every $g$ with $V_{g}\neq 0$.

\item

$p(Z_{G})$ is a $G$-graded nonidentity of $A$.

\item

If $\phi$ is \textit{any} nonzero evaluation of $p(Z_{G})$ on $A$
(with elements of the form $u_{h}\otimes e_{i,j}$!), then all
monomials but one vanish and for the unique monomial of $p(Z_{G})$
which does not vanish, say $Z_{G}$, the elements of the set
$S_{g}$ assume precisely all basis elements of $V_{g}$.

\end{enumerate}

Roughly speaking we construct alternating polynomials which are
not only nonidentities of $A$, but also have the property that by
means of  \textit{any} nonvanishing evaluation we are able to
allocate the elements of $V_{g}$, $g \in G$ to the variables in
$S_{g}$ (in the sense of (5) above). In this case we will also say
that the polynomial $p$ allocates the elements of $V_{g}$. The
upshot of this is that since $A$ and $B$ satisfy the same
$G$-graded identities, we will be able to allocate homogeneous
basis elements of $B$ determined by the presentation $P_{B}$.

In what follows we will show how to construct such polynomials for
certain $G$-graded subspaces $V$ of $A$ which correspond to the
invariants mentioned above.
 In order to construct the polynomials (roughly speaking) we
proceed as follows. We identify in the algebra $A$ (say) the
spaces $(V_{g})$ as well as the full $g$-component of $A$ for
any $g$ which appears as a homogeneous degree in the $V_{g}$'s
(no damage if we add a homogeneous degree $g$ for which
$V_g=0$). We write a \textit{nonzero} monomial with the basis
elements $u_{h}\otimes e_{i,j}$ where we pay special
attention to the spaces $V_{g}$'s.

For each basis element $u_{h}\otimes e_{i,j}$ which is to be part
of an alternating set we insert on its left the idempotent $1
\otimes e_{i,i}$ and on its right the idempotent $1 \otimes
e_{j,j}$. We refer to these idempotents as \underbar{frames}. Next
we consider the homogeneous degrees of the basis elements and we
construct a (long!) multilinear monomial, denoted by $Z_{G}$, with
homogeneous variables whose homogeneous degrees are as prescribed
by the just constructed monomial in $A$. Finally we alternate the
homogeneous sets of cardinality equal to the full dimension of the
$g$-homogeneous component in $A$.

We start with step (1),  the dimensions of the homogeneous
components. It is well known that there is a nonzero product of
the form
$$e_{1,1}\times e_{1,2}\times \cdots \times e_{i,1}=e_{1,1},$$
of all elementary matrices $e_{i,j}$, $1 \leq i,j \leq r$.
Clearly, for every $h\in H_{A}$, the product of the monomial
$\Sigma_{h} = u_{h} \otimes e_{1,1} \times u_{h} \otimes e_{1,2}
\times \cdots \times u_{h} \otimes e_{i,1}$ is nonzero and is of
the form $\lambda_{h}u_{h^{r^{2}}} \otimes e_{1,1}$, where the
scalar $\lambda_{h} \in F^{*}$ depends on the $2$-cocycle $\alpha$
on $H$. Clearly, the product $\Pi_{h} \Sigma_{h}$ of the monomials
$\Sigma_{h}$'s yields a nonzero product of the form
$\lambda_{H}u_{h_{0}} \otimes e_{1,1}$ for some $\lambda_{H} \in
F^{*}$ and some $h_{0} \in H$. Let us denote the entire product by
$\Sigma_{H}$. We refer to its elements as designated elements. Now
we insert frame elements of the form $1 \otimes e_{i,i}$ on the
left and and on the right of any basis element in $\Sigma_{H}$ so
that the entire product $\tilde{\Sigma}_{H}$ is nonzero. The key
property that we need here is that the pairs of indices $(i,j)$
and $(k,s)$, of any two different basis elements $u_{h} \otimes
e_{i,j}$ and $u_{h^{'}} \otimes e_{k,s}$ having the same
homogeneous degree, must be different. Consequently, if we permute
designated elements of $\tilde{\Sigma}_{H}$, of the same
homogeneous degree, we obtain zero.

Consider the homogeneous degree of all basis elements which appear
in $\tilde{\Sigma}_{H}$ and produce a (long) multilinear monomial
$Z_{G}$ whose elements are homogeneous of degrees as prescribed by
the elements of $\tilde{\Sigma}_{H}$. We denote variables which
correspond to designated basis elements by $z_{i,g}$'s and refer
to them as designated variables. Variables which correspond to
frame elements will be denoted by $y_{j,e}$'s. Note that by
construction, the number of designated variables of degree $g$
coincides with the dimension of $A_{g}$, for every $g \in G$. Now,
for every $g \in G$, we alternate the designated variables of
degree $g$ in $Z_{G}$ and denote the polynomial obtained by $p$.
By construction $p$ is a $G$-graded nonidentity of $A$ and so by
assumption it is also a $G$-graded nonidentity of $B$. But by the
alternation property of $p$ we have that $\dim(A_{g}) \leq
\dim(B_{g})$, for every $g \in G$, and so we are done by symmetry.
This completes the proof of step (1).

We proceed to step (2). Consider the $e$-component of
$A$. By Proposition \ref{structure of N-graded subalgebra} and Remark \ref{blocks and pages},
it is isomorphic to the direct sum of simple
algebras which can be realized in blocks along the diagonal. By
permuting the elements of the $r$-tuple (which provides the elementary grading) we can order the
$e$-blocks in decreasing order. Construct a monomial $Z_{G}$
with segments which pass through each one of the $e$-blocks,
bridged by an element (necessarily) outside the $e$-component. We
insert frames of idempotents around the elements of the $e$-blocks. The
prescribed sets $V$ here are determined as follows. For the
maximal size (say $d_1$) of $e$-blocks, we have $r_1$ blocks, for
the second size $(d_2)$ we have $r_2$ blocks and so on. So we have
$r_1$ $e$-spaces of the largest dimension $(d_1)^{2}$, and so on.
We produce the alternating polynomial as above.

\begin{proposition}

The polynomial above allocates the $e$-blocks, where the
$e$-blocks of the same dimension are determined up to permutation.
\end{proposition}

\begin{proof}

First note that the polynomial $p$ is a nonidentity of $A$. To
see this let us show that the evaluation (which determined the
monomial $Z_{G}$) is indeed a nonzero evaluation. Clearly the
monomial $Z_{G}$ does not vanish by construction. On the other
hand in any nontrivial alternation, elements of the $e$-blocks
will meet the wrong idempotent frames and so we get zero.

Next let us show that for any nonzero evaluation of the
polynomial (on the standard basis) we have
that all monomials but one vanish and for the one that does not
vanish, the evaluation allocates the $e$-blocks as prescribed.  Indeed, we note first that by
the alternation property we are forced to evaluate the full
$e$-alternating set by a full basis of the $e$-component (for
otherwise we get zero) so taking a basis of $e$-elements of the
form $u_{h}\otimes e_{i,j}$ we are forced to use all of them and
each exactly once.

Next we analyze the evaluation of any monomial whose value is
nonzero. Elements of the $e$-component that are substituted for
variables of the same segment must belong to the same block for
otherwise we obtain zero: Indeed, segments consist only of
$e$-variables and basis elements of different blocks can be
bridged only by (homogeneous) elements of degree $\neq e$. In
other words variables of any segment must be evaluated only by
elements of the same $e$-block. Consider a segment of largest
size. Since it must be evaluated by elements of one single block,
it must exhaust one of the blocks of size $d_{1}^{2}$. Proceeding
to the next segment we see that we must substitute elements from
the $e$-block of next largest (perhaps the same) size. Continuing
in this way we obtain the desired allocation property.

\end{proof}

Having constructed the polynomial $p$, we would like to see what can
be deduced from the fact that $p$ is also a nonidentity of the
$G$-graded algebra $B$.
Without loss of generality let us assume the the configuration of
the multiplicities (i.e. the sizes of the $e$-blocks) for $A$ is
larger than for $B$ (with the lexicographic order). It follows
that in the largest $e$-segment we must put a full $e$-block and
so we must have an $e$-block of the corresponding size in $B$.  Continuing
in this way  we see that the multiplicities in $B$
must be the same as in $A$ and so we have step (2).

For step (3),   note that because the size of the matrix part in
$P_A$ (resp. $P_{B}$) is the sum of the sizes of the $e$-blocks of
$P_{A}$ (resp. $P_{B}$),  the size of the matrix part of $P_A$ and
$P_B$ must be the same. But we have seen that  $A$ and $B$ have
the same dimension.   It follows that the subgroups $H_{A}$ and
$H_{B}$ have the same order.

Before we proceed to the next step  we want to return to the proof
of step (2) above and present here a proof which is basically the
same but with more precise notation (in contrast to the
``algorithmic'' style presented above). We find this second
presentation more cumbersome and so we do it (as an illustration)
only for this step.

Let $P_{A}=(H, \alpha, (p_1,\ldots,p_r))$ be the given
presentation of the algebra $A$. Applying basic moves we know that
elements $p_{i}$'s may be replaced by right $H$-cosets
representatives and so we write the $r$-tuple $(p_1,\ldots,p_r)$
as
$$(g_{(1,i_1)},g_{(2,i_1)},\cdots,g_{(d_1,i_1)}, g_{(1,i_2)},g_{(2,i_2)},\cdots,g_{(d_2,i_2)}, \cdots, g_{(1,i_m)},g_{(2,i_m)},\cdots,g_{(d_m,i_m)})
$$
where $g_{(k,i_s)}= g_{(l,i_s)}$ for all $s$,  $1\leq s\leq m$ and all $k$ and $l$ in $\{1,2,\dots , d_s\}$, and
$Hg_{(k,i_s)}\neq Hg_{(l,i_t)}$ for $s \neq t$. Clearly, $d_1 +
d_2 + \ldots + d_m = r.$

With this notation, the $e$-component is spanned by the basis
elements $u_{e} \otimes e_{i,j}$ where $d_1 + d_2 + \ldots +
d_{k-1}+1 \leq i,j \leq  d_1 + d_2 + \ldots + d_{k}$, $1\leq k
\leq m$. Furthermore, the $e$-component is decomposed into the
direct sum of $m$ simple algebras $A_i$, which are clearly
isomorphic to the matrix algebras $M_{d_i}(F)$.

It is well known that for each one of the simple algebras $A_{k}$
(of degree $d_k$) there is a nonzero product
$\overrightarrow{E}_{k}$ of precisely all basis elements, starting
with $1 \otimes e_{t,t}$ and ending with $1 \otimes
e_{s,t}$ where $t = d_1 + d_2 + \ldots + d_{k-1}+1$ and $s = d_1 +
d_2 + \ldots + d_{k}$. Next we expand each monomial
$\overrightarrow{E}_{k}$ by bordering every basis element $1
\otimes e_{s,t}$ which appears in it by idempotents $1\otimes
e_{s,s}$ and $1 \otimes e_{t,t}$ from left and right
respectively. We denote the monomial obtained by
$\overrightarrow{E}_{k,fr}$ (``fr'' stands for framed). We view
the basis elements $1\otimes e_{s,t}$ (of
$\overrightarrow{E}_{k}$) as ``designated'' elements (which are
about to alternate) and the idempotents $1\otimes e_{t,t}$ as
``frame'' elements. The product of basis elements
$\overrightarrow{E}_{k,fr}$ consists of designated elements as
well as frames. Note that all basis elements in
$\overrightarrow{E}_{k,fr}$ are homogeneous of degree $e$.

\begin{remark}

Note that in the nonzero product above of the basis elements we do
not insist (although it is possible here) that each basis element
appears precisely once but rather that it appears at least once.
In case we have repetitions we may include the additional basis
elements as part of the frame.
\end{remark}

Now consider basis elements which bridge the different blocks. For
instance the elements $a_{k,k+1}=1 \otimes e_{i,j}$, $(i,j)=
(d_1 + d_2 + \ldots + d_{k-1}+1, d_1 + d_2 + \ldots + d_{k})$,
$k=1,\ldots, m-1$, bridge the $k$-th and $k$-th $+1$ block
respectively. From the structure of the $r$-tuple we see that the
homogeneous degree of $a_{k,k+1}$ is $\neq e$ (say $g_{\pi_{k}})$.
We obtain that the product of basis elements
$T_{1}a_{1,2}T_{2}a_{2,3}\cdots a_{m-1,m}T_{m}$ is nonzero (in
fact the product is $1\otimes e_{1,d_1 + d_2 + \ldots +
d_{m-1}+1}$). We see that any non trivial permutation on the
designated basis elements (and leaving the other elements fixed)
gives a zero product.

Now we create the multilinear monomial $Z_{G}$. The designated
basis elements will be replaced by ``designated variables''
$z_{i,e}$ whereas the rest of the basis elements (frames and
bridges) will be denoted by $y_{j,g}$, where $g$ is the
corresponding homogeneous degree. To sum up, we have the
following. From each product of basis elements
$\overrightarrow{E}_{k,fr}$ we construct a multilinear monomial
$\overrightarrow{T}_{k,fr}$ of $e$-variables (designated variables
and frames). Then the monomial $Z_{G}$ is given by the product

$$
\overrightarrow{T}_{1,fr} y_{1} \overrightarrow{T}_{2,fr} y_{2}
\cdots y_{m-1} \overrightarrow{T}_{m,fr},
$$
where $y_i$ has weight $g_{\pi_i}$.   Finally, the polynomial $p(Z_{G})$ is obtained by alternating the
variables $z_{i,e}$ of $Z_{G}$. Note that the polynomial has
precisely $dim_{F}(A_{e})$ variables $z_{i,e}$.

Now we consider the degrees of the above $e$-blocks. Assume the
$e$-blocks are ordered with degrees in decreasing order, so $d_1
\geq d_2 \geq \ldots \geq d_m$. Let $\chi_{1}$ be the number of
blocks of degree $d_1$, $\chi_{2}$ the number of blocks of degree
$d_{\chi_{1}+1}$, and finally $\chi_{\nu}$ the number of blocks of
lowest degree.

In terms of the terminology above we have vector spaces
$V_{1,1,e}, \ldots, V_{1, \chi_1,e}$ which are the first
$\chi_{1}$ diagonal blocks, and are of dimension $d^{2}_{1}$ ,
$V_{2,1,e}, \ldots, V_{2, \chi_2,e}$ are the next $\chi_{2}$
diagonal blocks, and are of dimension $d^{2}_{2}$ and so on.

We claim that the polynomial $p(Z_{G})$ satisfies the allocation
property for the spaces $V_{i,j,e}$. Clearly, by construction,
$p(Z_{G})$ is a multilinear $G$-graded nonidentity of $A$.

Next we need to see that any nonzero evaluation with basis
elements allocates the vector spaces $V_{i,j,e}$ up to a
permutation of the second index. By construction, the polynomial
$p(Z_{G})$ alternates on the set of designated variables (of
degree $e$) whose cardinality equals the dimension of $A_{e}$ over
$F$. Consequently, in any nonzero evaluation, the designated
variables must assume precisely elements which form a basis of
$A_{e}$ and so, choosing (as we may) a basis of $A_{e}$ of the
form $ 1 \otimes e_{i,j}$, the designated variables must assume
each of these elements exactly once. Next we show that (in a
nonzero evaluation) each set of designated variables in
$\overrightarrow{T}_{k,fr}$ must assume precisely all basis
elements of a unique $e$-block. Indeed, basis elements of
different $e$-blocks cannot be bridged by $e$-homogeneous elements
and so designated variables $\overrightarrow{T}_{k,fr}$ must get
values from a unique $e$-block. But the cardinalities of the sets
of designated variables of $\overrightarrow{T}_{k,fr}$ coincide
with the dimensions of the different $e$-blocks and so the result
follows from the following obvious lemma.

\begin{lemma}

Let $\Omega_{A}$ and $\Omega_{\bar{A}}$ be finite collections of
finite sets $A_1, \ldots, A_n$ and $\bar{A}_1, \ldots, \bar{A}_n$
respectively. Assume the sets $A_{i}$, $i=1,\ldots,n$ are pairwise
disjoints (likewise for the $\bar{A}_{i}$'s).

Suppose the cardinality of $A_{i}$ and $\bar{A}_{i}$ coincide for
$i=1,\ldots,n$. Let $\mathcal{A}$ and $\mathcal{\bar{A}}$ be the
union of the $A_{i}$'s and the $\bar{A}_{i}$'s respectively.
Suppose

$$
\phi: \mathcal{A} \rightarrow \mathcal{\bar{A}}
$$
is a bijection such that any two elements of different $A_{k}$'s
(say $a_{f} \in A_{f}$ and $a_{h} \in A_{h}$, $f \neq h$) are
mapped to different $\bar{A}_{i}$.

Then there is a permutation $\pi \in S_{n}$ such that the map
$\phi$ establishes a bijection of $A_{i}$ with $\bar{A}_{\pi{i}}$
for $i=1,\ldots,n$. Furthermore, if (to begin with) the sets $A$'s
and $\bar{A}$'s are ordered in decreasing order, then the
permutation $\pi$ permutes only sets of the same order.

\end{lemma}

The rest of the argument is the same as for the first proof.

At this point we have that   the
multiplicities of the right $H_{A}$-coset representatives in the
$r$-tuple $(p_1,...,p_r)$  (of the presentation $P_{A}$) coincides
with the multiplicities of right $H_{B}$-cosets in the
corresponding tuple of $P_{B}$.   In particular the matrix size of
the presentations $P_{A}$ and $P_{B}$ coincide.  Moreover  the subgroups
$H_A$ and $H_B$ have the same order.

The next step, number $4$ of the outline, is to show that the
subgroups $H_A$ and $H_B$ are conjugate in $G$.   For this and
later steps we  introduce a polynomial that generalizes the one
above.   We may arrange the tuples of coset representatives for
$A$ and $B$ so that representatives of the same coset (of $H_A$ in
$G$ for $A$ and of $H_B$ in $G$ for $B$)  are grouped together and
so that we use the same group elements for the same coset.   We
have proved that the number of coset representatives (with
multiplicities) is the same for $A$ and $B$.   Now let $T$ be an
arbitrary subgroup of $G$.  We have seen that  $A_T$, the
$T$-component of $A$, is a sum of $T$--simple algebras that appear
in blocks in $A$.   For each block we produce a nonzero product of
the standard basis elements that lie in that block, each used
exactly once,  with the extra condition that the first part of the
product uses those standard basis elements in that block with
weight $e$.    In other words the product begins with a nonzero
product of the standard basis elements determined by  that part of
the $e$-component that lies in that block.  This part of the
$e$-component is a semisimple algebra.   For each simple component
we produce a nonzero product of the standard basis elements from
that component. We then add frames of weight $e$ between every
pair of these basis elements.   We then add frames  of weights in
$T$ but necessarily not of weight $e$ between these simple
components.    We then complete the product for the rest of the
standard basis elements in that block.  Finally we put these block
products in some order and between each pair of successive blocks
we put another  standard basis  element (necessarily of weight
outside of $T$)  so that the entire product is nonzero.   We now
form a monomial from this product.    Denote it by $Z_{T,A}$.  We
then alternate the variables of the same weight (in $T$)  that
came from the standard basis elements in each block.    Denote the
resulting polynomial  $f_{T,A}$.  We \underbar{claim} that this
polynomial is a $G$-graded nonidentity of $A$. Indeed, replacing
the monomial $Z_{T,A}$ with the original basis elements we obtain
a nonzero product. Let us show now that for any nontrivial
alternation,  some standard basis element   will be bordered by
elements which annihilate it. To see this note that two basis
elements with the same $(i,j)$ position cannot have the same
homogeneous degree. This shows that elements with equal
homogeneous degrees are bordered by basis elements of the form
$1\otimes e_{i,i}, 1\otimes e_{j,j}$ and $1\otimes
e_{i^{'},i^{'}}, 1\otimes e_{j^{'},j^{'}}$ where the pairs $(i,j)$
and $(i^{'},j^{'})$ are different. It follows easily that any
nontrivial alternation yields a zero value. This proves the claim.
Of course we can do the same thing for $B$ and we denote the
resulting polynomial  $f_{T,B}$.

\begin{proposition} \label{subgroup decomposition}

Let $T$ be a subgroup of $G$.   There is a one-to-one
correspondence between the $T$--simple components of $A_T$ and the
$T$--simple components of $B_T$  such that corresponding
components have the same dimension and same matrix size.
Moreover for corresponding components the vector of
multipliciities  of the coset representatives (from the tuple for
$A$ and the tuple for $B$) is the same.
\end{proposition}

\proof  Because $f_{T,A}$ is a nonidentity for $A$, it must be a
nonidentity for $B$, so some monomial $Z$ of $f_{T,A}$ must be
nonzero on $B$. (In fact because each of the monomials of
$f_{T,A}$ is an alternation of $Z$,   if $Z$ has some nonzero
evaluation so does every monomial, so we could assume
$Z=Z_{T,A}$).  Under the evaluation of $Z$ no two blocks of $B_T$
can be substituted into the segment coming from a single block of
$A_T$   because elements of different blocks annihilate each
other.    So consider the $T$--simple component of $B_T$ of
smallest dimension.    When we evaluate $Z$ on $B$ this block must
completely fill some segment.  In other words the dimension of
this smallest component must be at least as large as the dimension
of the smallest component of $A_T$.   Since we can do the same
argument for $f_{T,B}$ we infer that the dimension of this
smallest component is the same as the  dimension of the smallest
component of $A_T$.   Continuing with the component of next
smallest dimension and so on,  we see that we have a one-to-one
correspondence between the $T$--simple components of $A_T$ and the
$T$--simple components of $B_T$  such that corresponding
components have the same dimension.    Moreover we see that in any
nonzero evaluation of $Z$ on $B$ we must substitute elements from
a given $T$--simple component of $B_T$  into a segment coming from
a $T$--simple component of $A_T$ of the same dimension.

 Next we \underbar{claim} that under such a substitution the component of $B_T$  must involve the same number of elements of the tuple f or $B$ with the same multiplicities
as the component for $A_T$  in whose segment of $Z$ the component
of $B_T$ is placed. To see this label these corresponding
components  $U_A$  and $U_B$.   Under the nonzero evaluation of
$Z$   the elements of the $e$-component of $U_B$  must be
substituted in the first part of the segment, the part formed from
the the $e$-component of $U_A$, and must fill that part of the
segment.    In particular the dimension of the $e$-component of
$U_B$  must be greater than or equal to the dimension of the
$e$-component of $U_A$.  Because this is true for every component
of $A_T$ and we know the dimensions of $A_e$  and $B_e$ are the
same, we see that the dimension of the $e$-component of $U_A$
equals the dimension of the $e$-component of $U_B$. But in fact
more is true.   Each of these $e$-components is a semi-simple
(ungraded) algebra.  Under the evaluation we cannot substitute two
elements from different simple components of the $e$-component in
$U_B$ into a segment coming from a single simple component of the
$e$-component in $U_A$ because such elements annihilate each
other.   Therefore the number of simple components of the
$e$-component in $U_B$  must be no larger than the number of
simple components of the $e$-component in $U_A$.   But the total
number of simple components of $A_e$  is the same as the total
number of simple components of $B_e$.   Again because we have the
inequality for all components of $A_T$ we see that the number of
simple components of the $e$-component in $U_B$  must equal the
number of simple components of the $e$-component in $U_A$.
Finally since the dimension of each of the simple components of
the $e$-component of $U_B$ must be greater than or equal to the
dimension of the simple component  of the $e$-component of $U_A$
in which it is evaluated, we see that the dimensions of the
simple components of the $e$-component in $U_B$ must equal to the
dimensions of the simple components of the $e$-component in $U_A$.
(In other words the $e$-component of $U_A$ is isomorphic as an
$F$--algebra to the $e$-component of $U_B$).       But the sum of
the matrix sizes of the simple components of the $e$-component of
$U_A$  is the matrix size of $U_A$, so $U_A$ and $U_B$ have the
same matrix size.   Moreover the matrix sizes of the the simple
components of the $e$-component of $U_A$  are exactly the
multiplicities of the elements of the tuple for $A$  that appear
there, so these are the same for $U_B$. \qed

We can now complete steps 4 and 5 of the outline.  Let $H=H_A$. By
applying a basic move to the presentation for $A$ we may assume
that $e$ appears in the tuple and that it appears with the highest
multiplicity.   Call this multiplicity $d$.    In the algebra
$A_H$ there will then be an $H$-simple component of dimension
$d^2|H|$  coming from the single coset representative $e$ in the
tuple.  By the proposition there must be a simple component of
$B_H$ of the same dimension and matrix size  coming from a single
coset representative $g$ (say) of $H_B$ in $G$ that appears in the
tuple for $B$.  Because the matrix size is the same as the
multiplicity we see that $g$ has multiplicity $d$.   Hence the
dimension of the corresponding component is $d^2|H\cap
g^{-1}H_Bg|$.   So we must have $d^2|H|=d^2|H\cap g^{-1}H_Bg|$.
Hence $|H|= |H\cap g^{-1}H_Bg|$.  Because $H$ and $H_B$ have the
same cardinality it follows that $H=g^{-1}H_Bg$,  so $H_A$ and
$H_B$ are conjugate.   By applying a basic move we may assume that
$H_A=H_B$.   We will denote this common subgroup by $H$.   We also
have that  the multiplicities arising in each $H$--simple
component is the same (up to permutation of the blocks)  in $A$
and $B$.

We now proceed to steps 6 and 7.    We decompose the tuples for
$A$ and $B$ as described before step 5 of the outline. Let $g$ be
a coset representative of $N(H)$ in $G$  that appears in the tuple
for $A$.    By Proposition~\ref{subgroup decomposition} we  know
that there is a one-to-one correspondence between the
$g^{-1}Hg$--simple components of $A_{g^{-1}Hg}$ and the
$g^{-1}Hg$--simple components of $B_{g^{-1}Hg}$  such that
corresponding components have the same dimension and matrix size.
Moreover for corresponding blocks the vector of multiplicities of
the coset representatives (from the tuple for $A$ and the tuple
for $B$) is the same.   In particular because an element of
$N(H)g$ appears in the tuple we see that we have blocks coming
from a single coset representative.   As in the case where $g=e$
this implies that the same is true for $B_{g^{-1}Hg}$  and so an
element of $N(H)g$ also appears in the tuple for $B$.   In fact
again as in the case where $g=e$ we see that the number of tuple
elements for $A$ that lie in the coset $N(H)g$ is the same as for
the tuple for $B$ including multiplicities.   This proves step 6.
It also proves step 7.

\begin{remark}

Note that if $H$ is $e$ then all we have so far is that  the
multiplicities in the $r$--tuples for $A$ and $B$ are the same. In
particular $A$ and $B$ have the same matrix size.

\end{remark}

Our next goal (step 8)  is to show that the tuples of the elementary grading
in $A$ and in $B$ are obtained from one another by
multiplication on the left by a single element of $N(H)$. This will
lead to the situation where the groups $H_{A}$ and $H_{B}$ are
still the same and the tuples are the same. Then the final parameter we will  need to deal with
will be  the $2$-cocycle on the group $H$.

We  consider a (very) special case of the statement above, namely
where $H$ is $e$. We have the tuple for $A$ and based on it we can
construct the polynomial $f_{\{e\}, A}$.  Let us
recall the construction.  We consider the $e$-blocks arising from
the multiple representatives. We produce $e$-segments for each
block bridged by non-$e$-elements. We know that the monomial is a
nonidentity of $A$ and if we put frames we know that any
nontrivial permutation of the designated $e$-elements gives a zero
product of basis elements.  We construct a
monomial out of the product above and alternate the designated
variables.

We denote by $\sigma_{1},...,\sigma_{n}$ the distinct $H$-coset
representatives in the tuple for $A$ and by
$\tau_{1},...,\tau_{n}$ the distinct coset representatives in the
tuple for $B$. Note that,  because $H=\{e\}$, distinct coset
representatives just means distinct elements. Also, we remind the
reader that by previous steps, the vector of multiplicities of
$\sigma_{1},...,\sigma_{n}$ and $\tau_{1},...,\tau_{n}$ is the
same. Let $d_{1},...,d_{n}$ be the vector of multiplicities, which
we may assume are  in decreasing order.   By
Proposition~\ref{subgroup decomposition} a nonzero evaluation on
$B$ gives rise to a permutation $\pi$ on
$\{\tau_{1},...,\tau_{n}\}$  so that the segment for $\sigma_i$ is
being evaluated by the elements in the $\tau_{\pi(i)}$ block.  For
every pair $i,j\in \{1,2,\dots, n\}$ the elements that can bridge
between the $i$-th block and the $j$-th block must have weight
$\sigma_{i}^{-1}\sigma_{j}$.   It follows that the bridging
elements between the $ \pi(i)$ block and the $\pi(j)$ block have
the same weight  and so we obtain  the relations
$\sigma_{i}^{-1}\sigma_{j}=\tau^{-1}_{\pi(i)}\tau_{\pi(j)}$ for
all $i,j$.
 Rewriting these equations,  we see that for all $i,j\in
\{1,2,\dots, n\}$,
$\sigma_{j}\tau^{-1}_{\pi(j)}=\sigma_{i}\tau^{-1}_{\pi(i)}$ and so
all the elements  $\sigma_{i}\tau^{-1}_{\pi(i)}$ are the same. We
see then that for all $i\geq 1$,
$\sigma_i=(\sigma_{1}\tau^{-1}_{\pi(1)})\tau_{\pi(i)}$, and so we
have found an element $g\in G$ such that $\sigma_i=g\tau_i$ for
all $i$.    That ends the case where $H=\{e\}$.

\begin{remark}
Note that not every permutation $\pi$ is allowed. For instance, a
permutation that exchanges elements with different multiplicities
would lead to a contradiction. In other words we cannot substitute
an $e$-block of size $d_{i}$ with an $e$-block of size $d_{j}\neq
d_{i}$. It is important to note (as mentioned above) that if a
segment was determined by a block of size $d_{i}$, arising from an
element $\sigma_{i}$ (say) (with multiplicity $d_{i}$) then in any
nonzero evaluation on $B$ (or on $A$) the segment will assume
values precisely of one block arising from $\tau_{j}$ where
necessarily $d_{j}=d_{i}$. Nevertheless, for the proof, we only
need to know the existence of a permutation $\pi$ as above.
\end{remark}

In fact a similar argument will work when $H$ is normal in $G$.
Because we will use it in the general case,  when $H$ is not
necessarily normal,  we  outline the normal case here:

Applying  Proposition~\ref{subgroup decomposition}  to  $A_H$ and
$B_H$  we see that there is a permutation $\pi$ on
$\{\tau_{1},...,\tau_{n}\}$ with the requirement that for every
$i,j \in \{1,2,\dots, n\}$,  the corresponding bridging elements
must have the same weight.  The set of weights of possible
bridging elements from the $i$-th to the $j$-th block in this case
is  the set $\sigma_{i}^{-1}H\sigma_{j}$. Therefore the necessity
of having bridging elements of the same weight for passing from
the  $ \pi(i)$ block to the $\pi(j)$ means that  for every $i,j
\in \{1,2,\dots, n\}$ the intersection

$$
\sigma_{i}^{-1}H\sigma_{j}\cap \tau^{-1}_{\pi(i)}H\tau_{\pi(j)}
$$
is nonempty. But because the $\sigma$'s and $\tau$'s normalize
$H$, these two sets are in fact cosets of $H$, and so must be
equal.   It follows that there are elements $h_{j} \in H$, $j \in
\{1,2,\dots, n\}$, such that for all $j$
$$
\sigma_j=(\sigma_{1}\tau^{-1}_{\pi(1)})h_{j}\tau_{\pi(j)}
$$
To complete the argument in this case recall that by basic move
(2) we may replace any element of the tuple for $B$ by a different
representative of the $H$-coset (that is replace $\tau_{\pi(j)}$
by $h_{j}\tau_{\pi(j)}$). We see therefore that the tuple for $A$
is obtained from the tuple for $B$ by multiplying on the left by
an element from $N(H) (=G)$.

We can now consider the general case where the group $H$ is not
necessarily normal in $G$.  We decompose the tuple for $A$ into
subtuples coming from different $N(H)$-representatives in $G$.  We
will refer to these subtuples as ``big blocks". We know that the
multiplicities in each subtuple coincide. We construct a monomial
which corresponds to that configuration:  We start with the
monomials $Z_{g_i^{-1}Hg_i, A}$ constructed before,   where
$g_1=e, g_2, \dots, g_k$ are the distinct coset representatives of
$N(H)$ in $G$ appearing in the tuple for $A$ which we have shown
can be taken to be also the distinct coset representatives of
$N(H)$ in $G$ appearing in the tuple for $B$.   We then form the
product of these monomials bridging successive monomials with
variables whose weights allow a nonzero evaluation using the
standard basis elements for $A$.    Call this big monomial $Z_A$.
We then perform successive alternations of the variables of a
given weight appearing in each of the monomials $Z_{g_i^{-1}Hg_i,
A}$.   Call this polynomial $f_A$.   This is a nonidentity for $A$
and so must be a nonidentity for $B$.    So one of the monomials
of $f_A$ must be nonzero on the standard basis of $B$ and as we
saw in the proof of Proposition~\ref{subgroup decomposition} we
may assume this monomial is  $Z_A$.    In particular each of the
submonomials  $Z_{g_i^{-1}Hg_i, A}$ must have a nonzero
evaluation.  We denote by $\sigma_{i,k}$ a typical  representative
of the cosets of $H$ in $N(H)$ such that $\sigma_{i,k}g_i$ appears
in the tuple for $A$  and by $\tau_{i,m}$  a typical
representative of the cosets of $H$ in $N(H)$ such that
$\tau_{i,m}g_i$ appears in the tuple for $B$.    The nonzero
evaluation of $Z_A$  then produces a permutation  $\pi$  on the
tuple for $B$ that preserves the subtuples coming from each coset
representative $g_i$ of $N(H)$ in $G$ that takes a block
corresponding to the coset representative $\sigma_{i,k}g_i$ to the
block coming from the coset representative $\tau_{i,\pi(k)}g_i$.
As before the bridge between the $\sigma_{i,k}g_i$  block and the
$\sigma_{j,m}g_j$  block must also serve as a bridge between the
$\tau_{i,\pi(k)}g_i$  block and the $\tau_{j,\pi(m)}g_j$ block and
so the set

$$
g_i^{-1}\sigma_{i,k}^{-1}H\sigma_{j,m}g_{j}\cap g^{-1}_{i}\tau^{-1}_{i, \pi(k)}H\tau_{j,\pi(m)}g_{j}
$$
must be nonempty.    Canceling  we obtain
$$\sigma_{i,k}^{-1}H\sigma_{j,m}\cap \tau^{-1}_{i,
\pi(k)}H\tau_{j,\pi(m)} $$  is nonempty.      It follows that
there are elements $h_{j,m} \in H$, $j \in \{1,2,\dots, n\}$, such
that for all $j,m$
$$
\sigma_{j,m}=(\sigma_{1,1}\tau^{-1}_{1,\pi(1)})h_{j,m}\tau_{j,\pi(m)}
$$
To complete the argument in this case recall that by basic move
(2) we may replace any element of the tuple for $B$ by a different
representative of the $H$-coset (that is replace $\tau_{j,
\pi(m)}$ by $h_{j,m}\tau_{j,\pi(m)}$). We see therefore that the
tuple for $A$ is obtained from the tuple for $B$ by multiplying on
the left by an element from $N(H)$.

So by applying basic moves,  we may now assume  that the fine
gradings of $A$ and $B$ are determined by the same group $H$ and
the elementary grading is determined by the same $r$-tuple
$(p_{1},...,p_{r}) \in G^{(r)}$. We proceed now to show that the
cocycles $\alpha$ and $\beta$ may be assumed to be the same.

We start with the case where the grading is fine,  that is,  $A$
and $B$ are twisted group algebras. Before stating the proposition
recall (Aljadeff, Haile and Natapov\cite{AHN}) that the $T$-ideal
of $H$-graded identities of a twisted group algebra $F^{\alpha}H$
is generated as a $T$-ideal by the multilinear binomial identities of the form

$$
B(\alpha)=x_{i_1, h_1}x_{i_2, h_2} \cdots x_{i_s, h_s} - \lambda_{((h_1,...,h_s),
\pi)}x_{i_{\pi(1)}, h_{\pi(1)}}x_{i_{\pi(2)}, h_{\pi(2)}} \cdots x_{i_{\pi(s)}, h_{\pi(s)}},
$$
where

\begin{enumerate}

\item

$h_{i} \in H$, $i=1,...,s$

\item
$\pi \in Sym(s)$

\item

the products $h_{1}h_{2}\cdots h_{s}$ and $h_{\pi(1)}h_{\pi(2)}\cdots h_{\pi(s)}$ coincide in $H$

\item

$\lambda$ is a nonzero element (root of unity) $\in F$ determined
by the  $s$-tuple $h_{1},h_{2},\ldots, h_{s}$ and the permutation
$\pi$.

\end{enumerate}

\begin{remark}
In fact more is true. If we allow repetitions of the homogeneous
variables, the binomial identities obtained span the $T$-ideal of
$H$-graded identities as an $F$-vector space. However we will not
need this fact here.

\end{remark}

\begin{proposition}
Given twisted group algebras $F^{\alpha}H$ and $F^{\beta}H$, then
the cocycles are cohomologous if and only if the algebras satisfy
the same graded identities.

\end{proposition}
\begin{proof}

The idea of the proof appeared already in \cite{AHN},  where we considered
the particular case where the group $H$ is of central type and the
twisted group algebra $F^{\alpha}H$ is the algebra of $k \times k$-matrices over $F$ where
$ord(H)=k^{2}$.
However, the same construction holds in general. For the reader's
convenience, let us recall here the construction.

It is well known, by the universal coefficient theorem, that the cohomology group $H^{2}(H,F^{*})$ is
naturally isomorphic to $Hom(M(H),F^{*})$ where $M(H)$ denotes the Schur
multiplier of $H$. It is also well known that $M(H)$ can be
described by means of presentations of $H$, namely the Hopf
formula. Indeed, let $\Gamma=\Gamma\langle x_{h_1},\ldots,x_{h_m}
\rangle$ be the free group on the variables $x_{h_i}$'s where
$H=\{h_1,\ldots,h_m\}$. Consider the presentation

$$
\{1\} \rightarrow R \rightarrow \Gamma \rightarrow H \rightarrow
\{1\}
$$
where the epimorphism is given by $x_{h_i} \longrightarrow h_i$.

One knows that the Schur multiplier $M(H)$ is isomorphic to

$$
R\cap [\Gamma,\Gamma]/[R,\Gamma].
$$

Given a $2$-cocycle $\alpha$ on $H$ (representing $[\alpha] \in
H^{2}(H,F^{*})$) it determines an element of $Hom(M(H),F^{*})$ as
follows: Let $[z]$ be an element in $M(H)$ where $z$ is a
representing word in $R\cap [\Gamma,\Gamma]$. For each variable
$x_{h}$ consider the element $u_{h}$ in the twisted group algebra
$F^{\alpha}H$ representing $h$.  Then the value of $\alpha$ on $z$
is the root of unity which is the product in $F^{\alpha}H$ of
the elements $u_{h}$ (which correspond to the variables $x_{h}$ of
$z$). One knows that the value $[\alpha]([z])$ depends on the
classes $[\alpha] \in H^{2}(H,F^{*})$ and $[z] \in M(H)$ and not
on their representatives. Note that by the isomorphism of $H^{2}(H,F^{*})$) with $Hom(M(H),F^{*})$
we have that for two noncohomologous
$2$-cocycles $\alpha$ and $\beta$ there is $z \in R\cap
[\Gamma,\Gamma]$ with $\alpha(z) \neq \beta(z)$. Let us show now how $H$-graded polynomial identities come into play.

Let
$$
z=x_{h_{i_1}}^{\epsilon_1}x_{h_{i_2}}^{\epsilon_2}\cdots
x_{h_{i_r}}^{\epsilon_r}
$$
where $\epsilon_i = \{\pm 1\}$. Because  $z$ is in $R$ we have that
$h^{\epsilon_1}_{i_1}h^{\epsilon_2}_{i_2}\cdots h^{\epsilon_r}_{i_r}=e$  and because $z$ lies in
$[\Gamma,\Gamma]$, we have  that the sum of the exponents $\epsilon_i$
which decorate any variable $x_{h}$ which appears in $z$, is zero.

Our task is to construct out of $z$ and the value $\alpha(z) \in
F^{*}$ an $H$-graded binomial identity of the twisted group
algebra $F^{\alpha}H$. Pick any variable $x_{h}$ in $z$ and let
$n$ be the order of $h$ (in $H$). Clearly the commutator
$[x_{h}^{n},y]$, $y \in \Gamma$, is in $[R,\Gamma]$ and so multiplying $z$ (say on the left) with
elements $x_{h}^{n}$ and $x_{h}^{-n}$, and moving them (to the right) successively along $z$ by means of the relation
$[x_{h}^{n},y]$, we obtain a
representative of $[z]$ in $M(H)$ of the form

$$
z_{1}z^{-1}_{2}
$$
where the variables in $z_{1}$ and $z_{2}$ appear only with
positive exponents.

The binomial identity which corresponds to $z$ and $\alpha(z)$ is
given by

$$
Z_{1}-\alpha(z) Z_{2}
$$
where $Z_{i}$ is the monomial in the free $H$-graded algebra whose
variables are in one to one correspondence with the variables of
$z_{i}$.
We leave the reader the task to show that indeed $ Z_{1}-\alpha(z)
Z_{2} $ is a $G$-graded identity. Clearly, from the construction
it follows that twisted group algebras $F^{\alpha}H$ and
$F^{\beta}H$ satisfy the same $G$-graded identities if and only if
the cocycles $\alpha$ and $\beta$ are cohomologous. This completes
the proof of the proposition.
\end{proof}

\begin{remark}

\begin{enumerate}

\item

The binomial identity obtained above, say for $\alpha$, may not be
multilinear. In order to obtain a multilinear binomial identity
assume $x_{h}$ appears $k$-times in each monomial $Z_{i}$,
$i=1,2$. Then replacing the variables by $k$ different variables
$x_{1,h},\dots,x_{k,h}$ in each monomial (any order!) we obtain an $H$-graded (binomial) identity
which is on variables whose homogeneous
degree is $h$. Repeating this process for every $h \in H$ gives a
multilinear (binomial) identity.

\item

It follows that any two noncohomologous cocycles can be separated
by suitable binomial identities in the sense that for any ordered
pair, $(\alpha, \beta)$ (where $(\alpha \neq \beta)$) there is a
binomial $B(\widehat{\alpha}, \beta)$ which is an identity of
$\beta$ (abuse of language) and not an identity of $\alpha$.

\item

Assume $\beta_{1},\ldots,\beta_{k}$ are cocycles on $H$ which are
different from $\alpha$ (noncohomologous to $\alpha$). Then by
the previous Proposition there is a binomial identity
$B(\widehat{\alpha}, \beta_{i})$ of $\beta_{i}$ which is a
nonidentity of $\alpha$. Then if we take the product of these
binomials (with different variables), we see that the product is a
multilinear identity of any of the $\beta_{i}$'s and not an
identity of $\alpha$.  This follows from the fact that in the twisted group algebra the product of two nonzero homogeneous elements is nonzero.

\end{enumerate}

\end{remark}

We now come to an important lemma which is due to Yaakov Karasik,
in which we extend the preceding proposition  to algebras with presentations in which the elementary grading
is trivial.

\begin{lemma}\label{Karasik}
Let $A$ and $B$ be finite dimensional $G$-simple algebras with presentations $P_{A}$
and $P_{B}$ respectively. Suppose $P_{A}$ and $P_{B}$ are given by
$F^{\alpha}H\otimes M_{r}(F)$ and $F^{\beta}H\otimes M_{r}(F)$
respectively, both with trivial elementary grading on $M_{r}(F)$.
If $\alpha$ and $\beta$ are noncohomologous,  then there is an
identity of $A$ which is a nonidentity of $B$.

\end{lemma}

\begin{remark}

In case the group $G$ is abelian, this was proved by
Koshlukov and Zaicev \cite{KZ} using certain modifications of the standard
polynomial. However this approach (at least in its straightforward generalization) seems to fail for nonabelian
groups.

\end{remark}

\begin{proof}

As above let $B(\widehat{\alpha}, \beta)$ denote a binomial identity of $F^{\beta}H$ which is a nonidentity of $F^{\alpha}H$.
Then $B(\widehat{\alpha}, \beta)$ has the
form

$$
B(\widehat{\alpha}, \beta)=z_{h_1}z_{h_2} \cdots z_{h_s} - \lambda_{(\widehat{\alpha}, \beta,
(h_1,...,h_s), \pi)}z_{h_{\pi(1)}}z_{h_{\pi(2)}} \cdots
z_{h_{\pi(s)}}.
$$

Next, consider the Regev polynomial $p(X,Y)$ on $2r^{2}$ variables
(each of the sets $X$ and $Y$ consists of $r^{2}$ variables). It
is multilinear (of degree $2r^{2}$) and central on $M_{r}(F)$. Any
evaluation of $X$ or $Y$ on a proper subset of the $r^{2}$
elementary matrices $e_{i,j}$ yields zero whereas in case $X$ and
$Y$ assume the full set of elementary matrices the value is
central, nonzero (and hence invertible) matrix
(see~\cite{formanek}).

Now, for each variable $z_{h}$ of $B(\widehat{\alpha}, \beta)$ we
construct a Regev polynomial on $2r^{2}$ variables where we pick
one variable from $X$ (no matter which) and we determine its
homogeneous degree to be $h$. The rest of the $x$'s and all the
$y$'s in $Y$ are determined as variables of homogeneous degree
$e$. We denote the corresponding Regev polynomial by
$p_{h}(X_{r^2},Y_{r^{2}})$. Now, we consider a basis of the
algebra $F^{\alpha}H\otimes M_{k}(F)$ consisting of elements of
the form $u_{h}\otimes e_{i,j}$. Note that there are precisely
$r^{2}$ basis elements of degree $e$ and $r^{2}$ basis elements of
degree $h$. We see that if we evaluate the polynomial
$p_{h}(X_{r^2},Y_{r^{2}})$ with elements in $\{1\otimes e_{i,j},
u_{h}\otimes e_{i,j}\}_{i,j}$ the result will be zero if the
elementary matrix constituent of the basis elements is
\textit{not} the full set of $r^{2}$ matrices (either for $X$ or
for $Y$) and $u_{h}\otimes \lambda \Id$ otherwise. It follows that
if we replace every variable $z_{h}$ in $B(\widehat{\alpha},
\beta)$ by the Regev polynomial $p_{h}(X_{r^2},Y_{r^{2}})$ we
obtain a polynomial

$$
R(\widehat{\alpha},
\beta, r)
$$
which is an identity of $B$ and a nonidentity of $A$.
\end{proof}

Before we continue, recall that a big block of $M_{r}(F)$ is any
block which is determined by elements of the tuple
$\{p_1,...,p_r\}$ which belong to the same right $N(H)$-coset of
$G$. A subblock of a big block is called ``basic'' if it is
determined by all elements of the tuple $\{p_1,...,p_r\}$ which
belong to the same right $N(H)$-coset of $G$ \underbar{and} have
the same multiplicity.

Let $\sigma_{1},...,\sigma_{t}$ be coset representatives of the
cosets of $H$ in $N(H)$ and $g_{1}=e, g_{2},\ldots, g_{n}$ be
coset representatives of the right cosets of $N(H)$ in $G$. For
any element $g \in \{g_{1}=e, g_{2},\ldots, g_{n}\}$, we consider
the representatives (for the right cosets of $H$ in $G$) given by
$\sigma_{1}g,...,\sigma_{r}g$. Clearly these representatives
determine a big block. We refer to this set as the set of
representatives of the big block determined by $g$. Note that in
the elementary grading for $A$ or for $B$ these representatives
may appear with different multiplicities. Next we fix a subblock
of the big block determined by $g$ by fixing the coset
representatives $\sigma_{1}g,...,\sigma_{m}g$ of $H$ in $G$.

Now, each one of these representatives, say $\sigma_{k} g \in
N(H)g$, conjugates the cocycle $\alpha$ into a cocycle
$\alpha^{(\sigma_k g)^{-1}}$ on the group $H^{g^{-1}}=g^{-1}H{g}$
and we \underbar{claim} that the sets of cocycles on  $H^{g^{-1}}$
which are obtained by conjugating the cocycles $\alpha$ and
$\beta$ by representatives of one subblock are the same (with
multiplicities).

Consider the set of cocycles $(\alpha^{(\sigma_{1}g)^{-1}},\ldots,
\alpha^{(\sigma_{m}g)^{-1})}$ obtained in the algebra $A$.  For
every cocycle $\alpha^{(\sigma_{k}g)^{-1}}$ we choose a set of
representatives for all cohomology classes on  $H^{g^{-1}}$ which
are different from the class represented by
$\alpha^{(\sigma_{k}g)^{-1}}$  and denote this set by $S_{g,k}$.

For each  cocycle $\gamma$ in $S_{g,k}$ we may construct, by
Lemma~\ref{Karasik}, a  polynomial   which is an identity for
$F^{\gamma}H^{g^{-1}} \otimes M_{d}(F)$ (where the elementary
grading is trivial) but not for $F^{\alpha^{(\sigma_{k}g)^{-1}}}
H^{g^{-1}} \otimes M_{d}(F)$. Moreover, if we take the product of
these polynomials (using different sets of variables) we obtain a
multilinear polynomial
$R({g,\widehat{\alpha}^{(\sigma_{k}g)^{-1}},d})$ which is an
identity for  the algebra $F^{\gamma}H^{g^{-1}} \otimes M_{d}(F)$
for all $\gamma \in S_{g,k}$, and  is a nonidentity if
$\gamma=\alpha^{(\sigma_{k}g)^{-1}}$.

We now construct a monomial for the   $g=g_i$ big block.    We
begin with the monomial $Z_{g^{-1}Hg, A}$  we considered in step
8.    This monomial was constructed by considering the graded
simple components of $A_{g^{-1}Hg}$.   Each such simple component
has a group part that is a subgroup of $H^{g^{-1}}$.  We will
alter the segments that come from simple components in which the
group part is all of $H^{g^{-1}}$.   Such a component comes from a
single coset representative $\sigma g$ in $N(H)g$  (with its
multiplicity).   For each such components we change the
corresponding segment of $Z_{g^{-1}Hg, A}$ by inserting, between
the end of the segment and the bridge to the next segment,   the
polynomial  $R({g,\widehat{\alpha}^{(\sigma g)^{-1}},d})$  (with
new variables),   where $d$ is the multiplicity of the
representative $\sigma g$  (which is also the matrix size of the
simple component).    Denote this new monomial
$\tilde{Z}_{g^{-1}Hg, A}$ .  We now use alternation to produce a
multilinear polynomial which we will denote  $\tilde{f}_{g^{-1}Hg,
A}$.

By its construction $\tilde{f}_{g^{-1}Hg, A}$ is clearly a
$G$-graded nonidentity of $A$ and hence, by assumption, it is also
a $G$-graded nonidentity of $B$.   We may therefore assume that
there is a nonzero evaluation of the monomial
$\tilde{Z}_{g^{-1}Hg, A}$ on $B$.   By the proof of
Proposition~\ref{subgroup decomposition}   there is a one-to-one
correspondence between the simple components of the big block of
$A$ coming from the coset $N(H)g$ and the simple components of the
big block for $B$ corresponding to the coset $N(H)g$  such that in
the evaluation of the monomial the segments of the monomials
determined by the simple components of the big block of $A$ are
evaluated at the corresponding simple components of that big block
for $B$.    But because of the inserted polynomials, say
$R({g,\widehat{\alpha}^{(\sigma g)^{-1}},d})$, we \underbar{claim}
the evaluation can be nonzero only if the cocycle on $H^{g^{-1}}$
determined by the coset representative of $H$ which evaluates the
segment, say $\tau g$, is cohomologous to $\alpha^{(\sigma
g)^{-1}}$. This will say that the cocycles $\alpha^{(\sigma
g)^{-1}}$ and $\beta^{(\tau g)^{-1}}$ are cohomologous in
$Z^{2}(H^{g^{-1}},F^{*})$ as desired. To prove the last claim we
note that the $Y$ variables in Regev's polynomials which appear in
$R({g,\widehat{\alpha}^{(\sigma g)^{-1}},d})$  are $e$ variables
and their cardinality is $d^{2}$. Next, since the set is
alternating we must evaluate the $Y$ variables by linearly
independent element. The $X$ variables of a Regev polynomial which
appears in $R({g,\widehat{\alpha}^{(\sigma g)^{-1}},d})$  all but
one of degree $e$ and one variable is of homogeneous degree in
$H^{g^{-1}}$. At any rate all variables are $H^{g^{-1}}$ variables
and so they must come from one single $H^{g^{-1}}$-block. But if
this $H^{g^{-1}}$-block is not the same as determined by the
segment evaluation we get zero. The upshot of this is that the
$H^{g^{-1}}$-block of $B$ which evaluates the segment (in a
nonzero evaluation) must determine a cocycle which does not
annihilate $R({g,\widehat{\alpha}^{(\sigma g)^{-1}},d})$. In other
words the cocycle must be cohomologous to $\alpha^{(\sigma
g)^{-1}}$. This completes the proof that the cocycles appearing in
a sub-block of $A$ and the corresponding sub-block of $B$ must be
the same up to permutation.

We are now reduced to the situation where the
multiplicities of the cocycles appearing in each basic block for the
algebras $A$ and $B$ coincide. In particular we know that the
cocycles $\alpha$ and $\beta$ are conjugate by an element of
$N(H)$.

The final step will be to show that up to equivalence we may
assume $\alpha$ and $\beta$ are actually cohomologous.    To do
this we will  produce an element  $b$ in $N(H)$  such that left
multiplication of the $r$-tuple $(p_{1},...,p_{r})$ permutes the
representatives in each big block in such a way that it preserves
multiplicities and conjugates $\alpha$ to $\beta$. Then by our
basic moves, the presentations $P_{A}$ and $P_{B}$ will be
equivalent.  The argument is similar to the proof of step 8.

   We start with the monomials
$\tilde{Z}_{g_i^{-1}Hg_i, A}$ constructed above,     where $g_1=e,
g_2, \dots, g_s$ are the distinct coset representatives of $N(H)$
in $G$ appearing in the tuple for $A$ (and $B$).   We then form
the product of these monomials bridging successive monomials with
variables whose weights allow a nonzero evaluation using the
standard basis elements for $A$.    Call this big monomial
$\tilde{Z}_A$.   We then perform successive alternations of the
variables of a given weight appearing in each of the monomials
$\tilde{Z}_{g_i^{-1}Hg_i, A}$.   Call this polynomial
$\tilde{f}_A$.   This is a nonidentity for $A$ and so must be a
nonidentity for $B$.    So one of the monomials of $\tilde{f}_A$
must be nonzero on the standard basis of $B$ and as we saw in the
proof of Proposition~\ref{subgroup decomposition} we may assume
this monomial is  $\tilde{Z}_A$.    In particular each of the
submonomials  $\tilde{Z}_{g_i^{-1}Hg_i, A}$ must have a nonzero
evaluation.  We denote by $\sigma_{i,k}$ a typical  representative
of the cosets of $H$ in $N(H)$ such that $\sigma_{i,k}g_i$ appears
in the tuple for $A$ (and $B$).    The nonzero evaluation of
$\tilde{Z}_A$  then produces a permutation  $\pi$  on the tuple
for $B$ that preserves the subtuples coming from each coset
representative $g_i$ of $N(H)$ in $G$ and so  takes a block
corresponding to the coset representative $\sigma_{i,k}g_i$ to the
block coming from the coset representative $\sigma_{i,\pi(k)}g_i$.
We also know that the cocycles  $\alpha^{(\sigma_{i,k}g_i)^{-1}}$
and $\beta^{(\sigma_{i,\pi(k)}g_i)^{-1}}$ (and hence
$\alpha^{\sigma_{i,k}^{-1}}$  and
$\beta^{\sigma_{i,\pi(k)}^{-1}}$) are cohomologous.

  As before
the bridge between the $\sigma_{i,k}g_i$  block and the
$\sigma_{j,m}g_j$  block must also serve as a bridge between the
$\sigma_{i,\pi(k)}g_i$  block and the $\sigma_{j,\pi(m)}g_j$ block
and so the set

$$
g_i^{-1}\sigma_{i,k}^{-1}H\sigma_{j,m}g_{j}\cap g^{-1}_{i}\sigma^{-1}_{i, \pi(k)}H\sigma_{j,\pi(m)}g_{j}
$$
must be nonempty.    Canceling  we obtain
$$\sigma_{i,k}^{-1}H\sigma_{j,m}\cap \sigma^{-1}_{i,
\pi(k)}H\sigma_{j,\pi(m)} $$  is nonempty.     It follows that
there are elements $h_{j,m} \in H$, $j \in \{1,2,\dots, n\}$ and
$m$ running over the subscripts for the representatives in the
$j$-th big block,   such that
$$
\sigma_{j,m}=(\sigma_{1,1}\sigma^{-1}_{1,\pi(1)})h_{j,m}\sigma_{j,\pi(m)}
$$
Recall that by basic move (2) we may replace any element of the
tuple for $B$ by a different representative of the $H$-coset (that
is replace $\sigma_{j, \pi(m)}$ by $h_{j,m}\sigma_{j,\pi(m)}$). We
see therefore that the tuple for $A$ is obtained from the tuple
for $B$ multiplying on the left  by
$\gamma=(\sigma_{1,1}\sigma^{-1}_{1,\pi(1)})$.  This element is in
the normalizer so we are left with showing  $\beta^\gamma$ is
cohomologous to $\alpha$.   But we know that
$\alpha^{\sigma_{1,1}^{-1}}$  and $\beta^{\sigma_{1,\pi(1)}^{-1}}$
are cohomologous, so we are done.
 This completes the proof of the main theorem.

\section{Uniqueness}   We end the paper  with a proof,  without the use of identities,  of the uniqueness of the decomposition given in Theorem \ref{Bahturin-Sehgal-Zaicev}.

\begin{proposition}\label{uniqueness}
Let $A$ and $B$ be a simple $G$--graded algebras with
presentations $P_{A}=(H_A, \alpha, (p_1,\ldots,p_r))$ and
$P_{B}=(H_B, \beta, (q_1,\ldots,q_s))$, respectively.  If $A$ and
$B$ are $G$--graded isomorphic then $P_A$ and $P_B$ are
equivalent.
\end{proposition}

\begin{proof} Let $\phi:A\rightarrow B$ be a $G$-graded isomorphism.
Then $\phi$ must take the $e$--component of $A$ to the
$e$--component of $B$, that is, $\phi(A_e)=B_e$. So
$\phi|_{A_e}:A_e\rightarrow B_e$  is an  $F$--algebra isomorphism
between these two semisimple (ungraded) $F$--algebras. This
isomorphism will take each simple component of $A_ e$ onto a
simple component of $B_e$. In particular the dimensions of these
corresponding components will be the same. It follows that $r=s$,
the multiplicities of the $p_i's$ is the same as that of the
$q_i's$ and that by possibly rearranging the tuple for $B$ (a
basic move),  we may assume the $i$--th block of $A_e$ is sent to
the $i$-th block of $B_e$.   Also because $r=s$ the elementary
parts of $P_A$ and $P_B$ have the same dimension.  Because $A$ and
$B$ have the same dimension we infer that $H_A$ and $H_B$ have the
same cardinality.

The next step is to show that the subgroups $H_A$ and $H_B$ are
conjugate in $G$. By shifting the tuple of the elementary grading
of $A$ (basic move (3)) we may assume that in the tuple
$(p_1,...,p_r)$ the elements $p_i$ are listed in order of
nonincreasing multiplicity and that $p_1=e$.

Let $H=H_A$.  In view of the Proposition \ref{structure of
N-graded subalgebra} and Remark \ref{blocks and pages},  the
subalgebra  $A_{H}$ decomposes into a direct sum of $H$-simple
subalgebras where each $H$-simple is a diagonal block with a
certain number of pages. Note that at least one of the blocks (for
example the block coming from the representative $e$) is ``full''
in the sense that  it has the maximal possible number of pages
(namely, $|H|$).  In fact we get this maximal number precisely
when the coset representative is in the normalizer of $H$. Now we
consider $B_H=\phi(A_H)$.  If $H$ is not conjugate to $H_B$, then
there will be no block in the decomposition of $B_H$ in which we
get a full number of pages,  because the number of pages in each
block is   the cardinality of $q_i^{-1}H_Bq_i\cap H$ for some $i$.
But applying $\phi$ to one of the full blocks of $A_H$ will have
to give a full block in $B_H$, so we have a contradiction.
Therefore $H_A$ and $H_B$ are conjugate. Applying a basic move we
may assume $H_A=H_B$.

So at this point we have reduced to the case where  $H_A=H_B$,
$r=s$,  the multiplicities in the tuples are the same, and under
the isomorphism $\phi$ the $i$--th block of $A_e$ is sent to the
$i$-th block of $B_e$.   We show next that we can assume the
tuples for $P_A$ and $P_B$ are the same.

 Let $H=H_A(=H_B)$.
Let $m$ denote the number of distinct $H$--coset representatives
in each tuple and let  $d_1\geq d_2\geq d_3 \cdots\geq d_m$ denote
the multiplicities. Now consider the matrix units
$e_{i,j}(=1\otimes e_{i,j})$.   We look at the $F$--span of the
images $\phi(e_{11}), \phi(e_{2,2})\dots, \phi(e_{d_1d_1})$. This
is the $F$--span of commuting semisimple elements in the first
component of $B_e$ and so is conjugate in the first component to
the space of diagonal matrices there.   The conjugating element
$b_1$ is an element (of weight $e$) in the first block of $B_e$.
We do the same for all $m$  blocks of $A_e$, obtaining elements
$b_1,b_2,\dots, b_m$.   The sum $b=b_1+b_2+\cdots b_m$ is an
invertible element in $B_e$ and  the composite of $\phi$ with
conjugation by $b$  will take the space of diagonal elements of
$A_e$ to the space of diagonal elements of $B_e$.   Because $b$ is
an invertible element of weight $e$, conjugation by $b$ does not
change the presentation for $B$ at all.  So by possibly reordering
the elements in the tuple for $B$ we may assume
$\phi(e_{ii})=e_{ii}$ for all $i$,  $1\leq i\leq r$. It follows
that for all $i,j$, $1\leq i,j\leq r$,
$\phi(e_{ij})=\phi(e_{ii}e_{ij}e_{jj})=e_{ii}\phi(e_{ij})e_{jj}$.
Moreover if $t\not=i$, then $e_{tt}\phi(e_{ij})=0$ and if
$t\not=j$, then $\phi(e_{ij})e_{tt}=0$.   Hence we must have
$\phi(e_{ij})=\gamma_{ij}u_{h_{ij}}\otimes e_{ij}$,  for some
$h_{ij}\in H$ and $\gamma_{ij}\in F^\times$.  In particular,
letting $i=1$,  we see that
$\phi(e_{1j})=\gamma_{1j}u_{h_{1j}}\otimes e_{1j}$,  where
$\gamma_{11}=1$ and $h_{11}=e$. But the weight of $e_{1j}$ is
$p_j$ (recall that $p_1=e$) and the weight of
$\gamma_{1j}u_{h_{1j}}\otimes e_{1j}$ is $q_1^{-1}h_{1j}q_j$.
Hence for all $j$, $1\leq j\leq r$,  $p_j=q_1^{-1}h_{1j}q_j$.  In
other words $q_1p_j=h_{1j}q_j$. Also because the first block of
$B_e$ has a full number of pages (it corresponds to the first
block of $A_e$ which has a full number of pages),   the element
$a_1$ must normalize $H$.  So by applying a basic move of type 3
and one of type 1 we can replace $(q_1, \dots, q_r)$ by
$(p_1,\dots, p_r)$ without changing $H$.

Notice also at this point that $\phi$ takes the elementary part of
$P_A$ to the elementary part of $P_B$ (in fact for all $i,j$,
$\phi(e_{ij})$ is a nonzero constant multiple of $e_{ij}$) and so
takes the centralizer of the elementary part of $P_A$ to the
centralizer of the elementary part of $P_B$.   In other words
$\phi$ takes $F^\alpha H$ to $F^\beta H$  and is a graded
isomorphism between these two $G$--graded algebras.  It is easy to
see that it follows that $\alpha$ and $\beta$ are cohomologous, so
we are done.

\end{proof}


\begin{thebibliography}{99}


\bibitem{AHN} E. Aljadeff, D. Haile and M. Natapov, {\em Graded identities of matrix
algebras and the universal graded algebra},  Trans. Amer. Math.
Soc. {\bf 362} (2010), no. 6, 3125�-3147.

\bibitem{AB} E. Aljadeff and A. Kanel-Belov, {\em Representability and Specht problem for $G$-graded
algebras}, Adv. in Math. {\bf 225} (2010),  2391--2428.

\bibitem{AKar} E. Aljadeff and Y. Karasik, {\em Central polynomial of crossed product algebras}, Preprint

\bibitem{AKassel} E. Aljadeff and C. Kassel, {\em Polynomial identities and noncommutative versal
torsors}, Adv. Math.  {\bf 218} (2008), no. 5, 1453--1495.

\bibitem{BSZ} Yu. A. Bahturin, S. K. Sehgal and M. V. Zaicev, {\em Finite-dimensional simple graded algebras},
Sb. Math. {\bf 199} (2008), no. 7, 965--983.

\bibitem{DR} V. Drensky and M. Racine, {\em Distinguishing simple Jordan algebras by
means of polynomial identities}, Comm. Algebra 20 (2) (1992),
309--327.

\bibitem{formanek}  E. Formanek, {\em A conjecture of Regev about
the Capelli polynomial}, J. Algebra {\bf 109} (1987), 93-114.

\bibitem{KR} A. Kushkulei and Y. Razmyslov, {\em Varieties generated by irreducible
representations of Lie algebras}, Vestnik Moskov. Univ. Ser. I
Mat. Mekh. 5 (1983), 4--7 (in Russian).

\bibitem{KZ} P. Koshlukov, and M. Zaicev, {\em Identities and isomorphisms of graded simple algebras}, Linear Algebra
Appl. {\bf 432} (2010), no. 12, 3141--3148.


\bibitem{SZ} I. Shestakov and M. Zaicev,
{\em Polynomial identities of finite dimensional simple algebras},
Comm. Algebra {\bf 39} (2011), no. 3, 929-932.
\end{thebibliography}
\end{document}